\documentclass{amsart}
\usepackage[all]{xy}
\usepackage{amsfonts, amsmath, amsthm, amssymb, tikz, accents, bm}
\usepackage[capitalise]{cleveref}
\title[Stability for coloured configurations and symmetric complements]{Homological stability for coloured configuration spaces and symmetric complements}
\author{TriThang Tran}
\email{trt@ms.unimelb.edu.au}
\date{\today}

\makeatletter
\newtheorem*{rep@theorem}{\rep@title}
\newcommand{\newreptheorem}[2]{%
\newenvironment{rep#1}[1]{%
 \def\rep@title{#2 \ref{##1}}%
 \begin{rep@theorem}}%
 {\end{rep@theorem}}}
\makeatother

\newreptheorem{theorem}{Theorem}
\newreptheorem{lemma}{Lemma}

\newtheorem{theorem}{Theorem}[section]
\newtheorem{lemma}[theorem]{Lemma}
\newtheorem{corollary}[theorem]{Corollary}
\newtheorem{proposition}[theorem]{Proposition}
\newtheorem{definition}[theorem]{Definition}
\newtheorem{example}[theorem]{Example}
\newtheorem{conjecture}[theorem]{Conjecture}

\theoremstyle{remark}
\newtheorem{remark}[theorem]{Remark}

\newcommand{\del}{\partial}
\newcommand{\mc}{\mathcal}

\newcommand{\norm}[1]{\lVert#1\rVert}
\newcommand{\st}{\; \vert \;}

\newcommand{\conf}{\mathrm{Conf}}
\newcommand{\pconf}{\mathrm{PConf}}

\newcommand{\sym}{\Sigma}
\newcommand{\symp}{\mathrm{Sym}}

\newcommand{\stab}{\mathrm{stab}}
\newcommand{\lf}{\left\lfloor}
\newcommand{\rf}{\right\rfloor}

\begin{document}
\maketitle
\begin{abstract} We prove a homological stability theorem for certain complements of symmetric spaces. This is a variant of a conjecture by Vakil and Matchett Wood \cite[Conjecture F]{vw12} for subspaces of $\symp^nX$ where $X$ is an open manifold admitting a boundary. To do this we prove a homological stability result for a type of ``coloured" configuration space by adding points of the same colour.
\end{abstract}

\section{Introduction}
For a manifold $X$, the \emph{ordered configuration space of $X$} is
	\[ \pconf_n(X) := \{ (x_1, \ldots, x_n) \in X^n \st x_i \neq x_j \mbox{ for } i \neq j \}. \]
The symmetric group $\sym_n$ acts on this space by permuting $n$-tuples. The \emph{unordered configuration space} is the quotient
 	\[ \conf_n(X) := \frac{\pconf_n(X)}{\sym_n}. \]
By the work of McDuff and Segal in the 70's \cite{mcduff75, segal79, segal73} and more recently by Randal-Williams and Church \cite{orw13, church12}, we know that the spaces $\conf_n(X)$ satisfy a property known as \emph{homological stability} when $X$ is the interior of a manifold $\bar{X}$ with boundary. In particular, there are maps
	\[ H_*(\conf_n(X), \mathbb{Z}) \rightarrow H_*(\conf_{n+1}(X), \mathbb{Z}) \]
that are isomorphisms for $* \leq \frac{n}{2}$. For a statement of the theorem in its most general form, see \cite{orw13}. In this paper, we study a generalisation of these spaces called the \emph{coloured configuration space}.

For a vector $\vec{v} \in (\mathbb{Z}_{\geq 0})^k$, let $\vec{v}(i)$ denote the $i^{th}$ coordinate of $\vec{v}$. The coloured configuration space of $X$ is the quotient
	\[ \conf_{\vec{v}}(X) := \frac{ \pconf_{\norm{\vec{v}}}(X) }{\sym_{\vec{v}(1)} \times \cdots \times \sym_{\vec{v}(n)}}, \]
where $\norm{\vec{v}} := \sum_i \vec{v}(i)$. One thinks of $\conf_{\vec{v}}(X)$ as configurations of $X$ where the points have been coloured. Points of a different colour are distinguishable but points of the same colour are not. The ordered configuration space occurs  as a special case of a coloured configuration space where there is only one point of each colour whereas the unordered configuration space occurs when there is only one colour. 

\begin{remark} In the literature (e.g., \cite{arnold69}), the ordered configuration space is sometimes referred to as the coloured (or dyed) configuration space. Our coloured configuration spaces differ from this in the sense that it will be possible to have more than one point of the same colour.
\end{remark}

If $X$ is the interior of a manifold $\bar{X}$ with boundary, then one can define a `stabilisation map' 
	\[ \stab_i : \conf_{\vec{v}}(X) \rightarrow \conf_{\vec{v} + e_i}(X) \]
which adds a point near the boundary of $X$ of the $i^{th}$ colour. Our first theorem is the following homological stability theorem for increasing the number of points of a single colour, which we prove in \cref{sec: hstab coloured}.

\begin{reptheorem}{hstab coloured}  Let $X$ be a connected open manifold of dimension greater than one that is the interior of a compact manifold with boundary. For each $i$, the map 
	\[ (\stab_i)_*: H_*(\conf_{\vec{v}}(X), \mathbb{Z}) \rightarrow H_*(\conf_{\vec{v} + e_i}(X), \mathbb{Z}) \]
is an isomorphism for $* \leq \frac{\vec{v}(i)}{2}$.
\end{reptheorem}

\begin{remark} In \cite{church12}, Church proves a version of \cref{hstab coloured} rationally. While we have not seen the integral version of \cref{hstab coloured} in the literature, it is easily derivable from the usual homological stability proofs for configuration spaces. Indeed it should be possible to repeat the proof in \cite{orw13} for unordered configuration spaces, allowing for the points not of the colour we are stabilising by to have a ``free ride" along the proof.
\end{remark}

The purpose of this paper is to use \cref{hstab coloured} to answer a conjecture proposed by Vakil and Matchett Wood in \cite[Conjecture F]{vw12}.

Let $\lambda$ be a partition of $n$ and let $X$ be a manifold. Let $\symp^n X$ denote the $n^{th}$ symmetric product of $X$. Define $w_\lambda(X)$ to be the subspace of $\symp^n X$ with multiplicities prescribed by $\lambda$. For example if $\lambda = 1 + 4 + 4$ then $w_\lambda(X)$ is the subspace of $\symp^9(X)$ consisting of elements that have one point of multiplicity $1$ and two points of multiplicity $4$. Let $\bar{w}_\lambda(X)$ be the closure of $w_\lambda(X)$ in $\symp^n X$. We will be interested in studying the complement $\bar{w}_\lambda^c(X)$ of $\bar{w}_\lambda(X)$ in $\symp^nX$.

One can think of the space $\bar{w}^c_\lambda(X)$ as consisting of elements of $\symp^nX$ such that the partition of $n$ obtained from the multiplicity of the points is not as coarse as $\lambda$. The main theorem of \cref{sec: symcomp} and this paper is the following.

\begin{reptheorem}{hstab symcomp}   Let $X$ be a connected open oriented manifold of dimension greater than one that is the interior of a compact manifold with boundary. Let $\lambda$ be a partition of $n$. There are maps
	\[ \stab_*: H_*(\bar{w}^c_{1^j\lambda}(X); \mathbb{Q}) \rightarrow H_*(\bar{w}^c_{1^{j+1}\lambda}(X), \mathbb{Q}) \] 
which are isomorphisms for $* \leq \frac{j+n}{4} - \frac{1}{2}$. 
\end{reptheorem}

\subsection{Motivation}
In a recent paper \cite{vw12}, Vakil and Matchett Wood were motivated by their calculations with motivic zeta functions to arrive at a number of conjectures. They call one such conjecture ``motivic stabilisation of symmetric powers" which says the following: for $X$ a geometrically irreducible variety of dimension $d$, the limit $\lim_{n \rightarrow \infty} [\symp^n X]/[\mathbb{A}^{dn}]$ exists in $\widehat{\mc{M}}_{\mathbb{A}^1}$, where $\widehat{\mc{M}}_{\mathbb{A}^1}$ is the Grothendieck ring of varieties localised at $[\mathbb{A}^1]$ and completed with respect to the dimension filtration on $\widehat{\mc{M}}_{[\mathbb{A}^1]}$. We leave the description of this to their paper. 

The spaces $\bar{w}_{\lambda}$ are subspaces of $\symp^nX$. For $\Phi$ a certain motivic measure, Vakil and Matchett Wood were able to make sense of the limit 
	\[ \lim_{j \rightarrow \infty}\frac{[\bar{w}_{1^j\lambda}]}{[\symp^{j + n}(X)]} \]
 in $\Phi(\hat{\mc{M}}_{[\mathbb{A}^1]})$ as a motivic probability. When put into a topological setting, this lead to the following conjecture.

\begin{conjecture}[{\cite[Conjecture F]{vw12}}] Given $i$ and a partition $\lambda$, for an irreducible smooth complex variety $X$, the limit $\lim_{j \rightarrow \infty} \dim{H_i(\bar{w}^c_{1^j\lambda}(X), \mathbb{Q})}$ exists.
\end{conjecture}

\cref{hstab symcomp} gives a positive answer to the analog of this conjecture for $X$ an open manifold of dimension greater than one admitting a boundary and any partition $\lambda$. The reason for studying the $\bar{w}^c_{\lambda}(X)$ as opposed to $\bar{w}_\lambda(X)$ is that in some sense the $\bar{w}^c_{\lambda}(X)$ behave nicer. For example, they are open subpaces of the orbifold $\symp^n(X)$, which means they satisfy rational Poincar\'{e} duality. In fact, in \cref{hstab symcomp}, we prove integral stability for the compactly supported cohomology of the spaces $\bar{w}_{1^j\lambda}^c(X)$ and only need to work rationally when we use Poincar\'{e} duality to convert back into homology.

\begin{remark} We would like to point out that as this paper was being prepared, a proof of \cite[Conjecture F]{vw12} in certain cases has appeared in a paper by Kupers and Miller \cite{km13}. In their paper they use topological chiral homology and prove the conjecture for open manifolds admitting a boundary for partitions where $\lambda = (c+1)^{m+1},$ for constants $c$ and $m$. They also prove a version for closed manifolds, which in this case was for partitions of the form $\lambda = c+1$. 

We have also been recently made aware of the fact that Kupers and Miller also have a preprint \cite{km13b} where they prove \cite[Conjecture F]{vw12} for $X$ being any connected manifold of dimension greater than 1.
\end{remark}

\subsection{Methods of proof}  In order to prove \cref{hstab coloured}, we study the fibre bundle
	\[ \conf_{\vec{v}}(X) \rightarrow \conf_{\vec{v}(m)}(X), \]
where $m \neq i$ is a different colour to the colour we are stabilising by. The fibre of this fibre bundle consists of coloured configurations with one fewer colour. The fibre will then satisfy homological stability by induction so it is possible to use the Serre spectral sequence to derive homological stability for the total space. 

In \cref{sec: symcomp}, \cref{hstab coloured} is used to study the spaces $\bar{w}^c_\lambda(X)$. We do this by showing that the spaces $\bar{w}^c_\lambda(X)$ are stratified spaces where the strata are coloured configuration spaces. We prove \cref{hstab symcomp} by studying the spectral sequence associated to the filtration of $\bar{w}^c_\lambda(X)$ by its strata. The idea is to carry a homological stability statement about each strata through the spectral sequence to a homological stability statement about the integral compactly supported cohomology of $\bar{w}^c_\lambda(X)$. It is in the translation back to homology that we must work rationally since the spaces $\bar{w}^c_\lambda(X)$ are only rational manifolds so they only satisfy rational Poincar\'{e} duality.


\subsection*{Acknowledgements}  
I would like to thank my advisor Craig Westerland for pointing me towards the conjectures of Vakil and Matchett Wood. I further thank him and Nathalie Wahl for many useful discussions relating to this work. I would also like to thank Alexander Voronov for providing me with some useful references and Jeremy Miller and Sander Kupers for sharing their preprint with me.

\section{Homology of coloured configurations} \label{sec: hstab coloured}

The \emph{$n^{th}$ ordered configuration space} of a manifold $X$ is the space of $n$ distinct points in $X$. In symbols it is
	\[ \pconf_n(X) := \{ (x_1, \ldots, x_n) \in X^n \st x_i \neq x_j \mbox{ for } i \neq j\}. \]
The symmetric group $\sym_n$ acts on $\pconf_n(X)$ by permuting $n$-tuples. The \emph{$n^{th}$ unordered configuration space} of $X$ is the quotient
	\[ \conf_n(X) := \frac{\pconf_n(X)}{\sym_n}. \]

In this section we will prove homological stability for certain types of coloured configuration spaces.

\subsection{The definition}
Let $\vec{v} \in (\mathbb{Z}_{\geq 0})^k$ for some $k > 0$. We will write $\vec{v}(i)$ for the $i^{th}$ coordinate of $\vec{v}$. Define $\norm{\vec{v}} := \sum_i \vec{v}(i) \in \mathbb{Z}_{\geq 0}$. 

\begin{definition} Let $X$ be a manifold. Let $\vec{v} \in (\mathbb{Z}_{\geq 0})^k$. The \emph{coloured configuration space of $X$} is the quotient of $\pconf_{\norm{\vec{v}}}(X)$ by the product $\sym_{\vec{v}(1)} \times \ldots \times \sym_{\vec{v}(n)}$ of symmetric groups. That is
	\[ \conf_{\vec{v}}(X) := \frac{ \pconf_{\norm{\vec{v}}}(X) }{\sym_{\vec{v}(1)} \times \cdots \times \sym_{\vec{v}(n)}}. \]
\end{definition}

One way to think of these configurations spaces is to think of the usual unordered configuration space with points grouped together by colours. The number of points of each colour is prescribed by the vector $\vec{v}$. For example, if $\vec{v} = (2, 3)$ then there are $2$ points of colour one, and $3$ points of colour two. More generally, there are $\vec{v}(i)$ points of colour $i$. Quotienting out by the product of symmetric groups means that we can distinguish points of different colours but cannot distinguish points of the same colour. These spaces lie somewhere in between ordered and unordered configuration spaces. 

\begin{remark} It is possible that some entries in $\vec{v}$ are $0$. Having extra $0$'s does not change the space. In \cref{sec: symcomp}, being able to have $0$ of a certain colour will be conceptually useful when we think of colours as multiplicities of points.
\end{remark}

\subsection{Stabilisation maps} \label{sec: stab maps}
Let $e_i$ be the vector with $1$ in the $i^{th}$ entry and $0$'s in the other entries. In the same way as one is able to define a stabilisation map $\stab : \conf_n(X) \rightarrow \conf_{n+1}(X)$ when $X$ is the interior of a manifold $\bar{X}$ with boundary (see e.g. \cite{orw13, palmer11}), we can, for each $i$, define stabilisation maps
	\[ \stab_i : \conf_{\vec{v}}(X) \rightarrow \conf_{\vec{v} + e_i}(X) \]
that add a point of colour $i$. 

Let $X$ be the interior of a manifold $\bar{X}$ with boundary $\del \bar{X}$. Pick a point $b_0 \in \del \bar{X}$ and let $\del_0$ be the component of $\del \bar{X}$ such that $b_0 \in \del_0$.
On the level of ordered configuration spaces, we define a map 
	\[ s :\pconf_{\norm{\vec{v}}}(X) \rightarrow \pconf_{\norm{\vec{v}} + 1}(X'), \]
where $X'$ is obtained from $\bar{X}$ by adding a collar neighbourhood around $\del_0$. The map $s$ is then given by inserting $b_0$ in front of the $ (\vec{v}(1) + \cdots + \vec{v}(i))$th coordinate. Picking an isomorphism $f: X \cong X'$, with support in a small neighbourhood of $\del_0$, we get a map 
	\[ f \circ s :\pconf_{\norm{\vec{v}}}(X) \rightarrow \pconf_{\norm{\vec{v} + e_i}}(X). \]
The map $\stab_i$ is then obtained as the quotient
	\[ \stab_i : \frac{\pconf_{\norm{\vec{v}}}(X)}{\sym_{\vec{v}(1)} \times \cdots \times \sym_{\vec{v}(i)} \times \cdots \times \sym_{\vec{v}(n)}}  \rightarrow \frac{\pconf_{\norm{\vec{v} + e_i}}(X)}{\sym_{\vec{v}(1)} \times \cdots \times \sym_{\vec{v}(i) + 1} \times \cdots \times \sym_{\vec{v}(n)}}. \]

\subsection{Homological stability for coloured configuration spaces}
In this section we will show that the stabilisation maps defined in \cref{sec: stab maps} induce isomorphisms in homology in a range. 

\begin{lemma} \label{hinj coloured} Let $X$ be a connected open manifold of dimension greater than one that is the interior of a compact manifold with boundary. The map 
	\[ (\stab_i)_*: H_*(\conf_{\vec{v}}(X), \mathbb{Z}) \rightarrow H_*(\conf_{\vec{v} + e_i}(X), \mathbb{Z}) \]
is always monic.
\end{lemma}
\begin{proof}
By symmetry it is enough to prove the result for $(\stab_1)_*$. We will use Lemma 2 of \cite{Dold62}. Let $\symp^n(X)$ denote the $n^{th}$ symmetric product of $X$ and recall the Dold-Thom theorem that $\pi_*(\symp^\infty(X)) \cong H_*(X)$. We will show that the induced stabilisation maps
	\[ (s_{\vec{v}(1)}) : \pi_*(\symp^\infty \conf_{\vec{v}}(X)) \rightarrow \pi_*(\symp^\infty \conf_{\vec{v}+e_1}(X)) \]
are injective for all $\vec{v}(1) \in \mathbb{Z}_{\geq 0}$.
Define 
	\[P_{k,\vec{v}(1)} : \conf_{\vec{v}}(X) \rightarrow \symp^{\vec{v}(1) \choose k} \conf_{\vec{v} - (\vec{v}(1) + k)e_1}(X) \]
	 to be the map that on the points coloured by the first colour is the map 
		\[ (p_1, \ldots, p_{\vec{v}(1)}) \mapsto \prod_{ \{p_{i_1}, \ldots , p_{i_k} \} \subset \{ p_1, \ldots, p_{\vec{v}(1)} \}} (p_{i_1}, \ldots, p_{i_k}), \]
and leaves the points of other colours the same.
Define 
	\[ \tau_{k,\vec{v}(1)} : \pi_* (\symp^\infty \conf_{\vec{v}}(X)) \rightarrow \pi_* (\symp^\infty \conf_{\vec{v} - (\vec{v}(1) + k)e_1}(X))\]
 as the induced map of
	\begin{align*} \symp^\infty \conf_{\vec{v}}(X) \xrightarrow{\symp^\infty P_{k,\vec{v}(1)}} \symp^\infty ( \symp^{ \vec{v}(1) \choose k}& \conf_{\vec{v} - (\vec{v}(1) + k)e_1}(X))  \\
						& \longrightarrow \symp^\infty \conf_{\vec{v} - (\vec{v}(1) + k)e_1}(X).
	\end{align*}
Here the last map is the map from $\symp^\infty \symp^{ \vec{v}(1) \choose k} \rightarrow \symp^{\infty}$ given by
	\[ (\bm{a}, \bm{b}, \ldots) \mapsto ( a_1, \ldots, a_{ \vec{v}(1) \choose k}, b_1, \ldots, b_{ \vec{v}(1) \choose k}, \ldots) \]
where $\bm{a} = (a_1, \ldots, a_{ \vec{v}(1) \choose k}) \in \symp^{ \vec{v}(1) \choose k}$, $\bm{b} = (b_1, \ldots, b_{{ \vec{v}(1) \choose k}}) \in \symp^{ \vec{v}(1) \choose k}, \ldots$
%

Clearly $\tau_{\vec{v}(1),\vec{v}(1)} = id$ and for $0 \leq k \leq \vec{v}(1) $ we have $\tau_{k,\vec{v}(1) + 1} s_{\vec{v}(1)} \equiv \tau_{k,\vec{v}(1)} \mod{im( s_{k-1})}$. Thus by Lemma 2 of \cite{Dold62}, the stabilisation map $s_k$ is split injective for $k \geq 0$.
\end{proof}

\begin{theorem} \label{hstab coloured}  Let $X$ be a connected open manifold of dimension greater than one that is the interior of a compact manifold with boundary. For each $i$, the map 
	\[ (\stab_i)_*: H_*(\conf_{\vec{v}}(X), \mathbb{Z}) \rightarrow H_*(\conf_{\vec{v} + e_i}(X), \mathbb{Z}) \]
is an isomorphism for $* \leq \frac{\vec{v}(i)}{2}$.
\end{theorem}

\begin{proof} 
By \cref{hinj coloured}, it is enough to show that the homology of the homotopy cofibre
 \[ H_*( \conf_{\vec{v} + e_i}(X), \conf_{\vec{v}}(X) ) = 0 \]
for $2* \leq \vec{v}(i)$. By symmetry, we only need to prove the theorem for stabilising by adding $e_1$.

We induct on the number of colours with a positive number of points. If there is only one colour, this is the usual homological stability for configuration spaces where the stability range is $* \leq n/2$ \cite{segal73, mcduff75, segal79, orw13}.

If there are $m \geq 2$ colours, we have the following commutative diagram of fibrations.
\[ \xymatrix{
	\conf_{\vec{v} - \vec{v}(m).e_m}(X_{m}) \ar[r] \ar[d]^{\stab_i} & \conf_{\vec{v}}(X) \ar[r]^{\pi} \ar[d]^{\stab_i} & \conf_{\vec{v}(m)}(X) \ar[d]^{=} \\
	\conf_{\vec{v} - \vec{v}(m).e_m + e_1}(X_{m}) \ar[r] & \conf_{\vec{v} + e_1}(X) \ar[r]^{\pi'} & \conf_{\vec{v}(m)}(X)
} \]

We use $X_m$ to denote $X$ with $m$ points removed. Since $X$ is connected, the choice of the $m$ points will not matter. The fibrations $\pi$ and $\pi'$ are the maps that forget all points not coloured by the $m^{th}$ colour. The middle and left maps down are the stabilisation maps. We therefore have a relative Serre spectral sequence with $E^2$ term
	\[E^2_{pq} = H_p ( \conf_{\vec{v}(m)}(X) , H_q( \conf_{\vec{v} - \vec{v}(m).e_m + e_1}(X_{m}), \conf_{\vec{v} - \vec{v}(m).e_m}(X_{m}) ) \]
converging to
	\[ E^\infty_{pq} = H_{p+q} ( \conf_{\vec{v} + e_1}(X), \conf_{\vec{v}}(X)). \]
By induction, 
	\[ H_q( \conf_{\vec{v} - \vec{v}(m).e_m + e_1}(X_{m}), \conf_{\vec{v} - \vec{v}(m).e_m}(X_{m}) ) = 0\]
 for $q \leq {\vec{v}(1)/2}$. Thus 
 	\[H_* ( \conf_{\vec{v} + e_1}(X), \conf_{\vec{v}}(X)) = 0 \] 
for $ * \leq \vec{v}(1)/2$. 
\end{proof}

%
%

\section{Stabilisation of symmetric complements} \label{sec: symcomp}
Given a partition $\lambda$ of $n$ and a manifold $X$, define $w_\lambda(X)$ to be the subspace of $\symp^n X$ with multiplicities prescribed by $\lambda$. For example if $\lambda = 1 + 4 + 4$ then $w_\lambda(X)$ consists of one point of multiplicity $1$ and two points of multiplicity $4$. Let $\vec{\lambda}$ be the vector whose $i^{th}$ entry corresponds to the number of times $i$ appears in the partition $\lambda$. In the previous example, $\vec{\lambda} = (1,0,0,2)$. 

\begin{proposition} \label{colour comp} $w_\lambda(X) \cong \conf_{\vec{\lambda}}(X)$.
\end{proposition}
\begin{proof}
Define $f: \conf_{\vec{\lambda}}(X) \rightarrow w_\lambda(X)$ by sending a point in $\conf_{\vec{\lambda}}(X)$ to the point in $w_\lambda(X)$ having $1$ point of each of the first colour, $2$ points of each of the second colour, and so on. The multiplicities of the new points will be given by $\lambda$ so the configuration lands in $w_\lambda(X)$. 
The inverse $f^{-1}: w_\lambda(X) \rightarrow \conf_{\vec{\lambda}}(X)$ is the map that removes multiple points, but remembers this by colouring multiplicity $i$ points by the $i^{th}$ colour.
\end{proof}

We have shown that the spaces $\conf_{\vec{\lambda}}(X)$ satisfy homological stability for increasing the number of points of a given colour. Hence the spaces $w_\lambda(X)$ also satisfy homological stability for increasing the number of points of a given multiplicity. We are primarily interested in increasing the number of multiplicity $1$ points.

Denote by $\bar{w}_\lambda(X)$ the closure of $w_\lambda(X)$ in $\symp^n X$. We will be interested in studying the complement $\bar{w}_\lambda^c(X)$ of $\bar{w}_\lambda(X)$ in $\symp^nX$. One thinks of the $\bar{w}_\lambda^c(X)$ as the subspace of $\symp^nX$ consisting of elements where the partition of $n$ given by the multiplicity of points is not a subpartition of $\lambda$. 

For a partition $\lambda$ of $n$, we define a \emph{subpartition of $\lambda$} to be a partition $\lambda'$ of $n$ that can be obtained from $\lambda$ by adding together the entries of the partition. For example, if $\lambda = 2 + 2 + 3 + 4$ then $\lambda' = 2 + 3 + 6$ is a subpartition of $\lambda$ obtained by adding $2$ and $4$. We use the notation $\lambda' \leq \lambda$ for subpartition.

For brevity the space $X$ will  sometimes be omitted from our notation.

\begin{example} $\bar{w}^c_{1^j2}$ is the unordered configuration space $\conf_{j+2} X$, where $1^j2$ is the notation we use for the partition $1 + \cdots + 1 + 2$ where we have added $1$ $j$ times. This is because $\bar{w}^c_{1^j2}$ is the subset of $\symp^{j+2}(X)$ with only multiplicity one points. Thus an element of $\bar{w}^c_{1^j2}$ consists of elements $\{ x_1, \ldots, x_{j+2} \}$ of $X$ with $x_i \neq x_j$.
\end{example}

The main theorem of this section is a rational homological stability theorem for the spaces $\bar{w}^c_\lambda(X)$, for certain $X$.
\begin{theorem} \label{hstab symcomp} Let $X$ be a connected open oriented manifold of dimension greater than one that is the interior of a compact manifold with boundary. Let $\lambda$ be a partition of $n$. There are maps
	\[ \stab_*: H_*(\bar{w}^c_{1^j\lambda}(X); \mathbb{Q}) \rightarrow H_*(\bar{w}^c_{1^{j+1}\lambda}(X); \mathbb{Q}) \] 
which are isomorphisms for $* \leq \frac{j+n}{4} - \frac{1}{2}$. 
\end{theorem}

\begin{remark}
In \cite{km13}, Kupers and Miller prove homological stability for topological chiral homology. As a consequence they prove integral homological stability for the spaces $\bar{w}^c_{1^j\lambda}$, where $\lambda = (c+1)^{m+1}$, $c, m \in \mathbb{Z}_{\geq 1}$. Our result differs by proving the theorem for all $\lambda$, though we only obtain a rational result. In the case of symmetric complements, our stability range is an improvement on their stability range of $* \leq \lf \frac {j-(m+1)(c+1)+1}{2c} \rf$.
\end{remark}

\cref{hstab symcomp} implies in particular the following which is a version of a conjecture by Vakil and Matchett Wood \cite[Conjecture F]{vw12} for open manifolds.

\begin{corollary} For a connected open oriented manifold $X$ of dimension greater than one that is the interior of a compact manifold with boundary, the betti numbers $b_i(j) := \dim H_i(\bar{w}^c_{1^j \lambda}(X) ; \mathbb{Q})$ stabilise as $j \rightarrow \infty$ i.e., the limit $\lim_{j \rightarrow \infty} b_i(j)$ exists.
\end{corollary}

The way we will prove \cref{hstab symcomp} is to study a stratification of $\bar{w}_\lambda^c$ into strata that will be disjoint unions of coloured configuration spaces. Homological stability holds for the strata so the key idea will be to transport this to a theorem on the whole space. To do this we will use the spectral sequence for the filtration given by the stratification and a version of Poincar\'{e} duality.

\subsection{Duality for stratified spaces} \label{ssec: duality}

Let $W$ be a space. A \emph{stratification} of $W$ is a finite filtration
\[ W = F_0 \supset F_1 \supset \ldots \supset F_l = \emptyset \]
such that $F_{p+1}$ is closed in $F_p$ and the filtration differences $M[p] :=F_p - F_{p+1}$ are manifolds (that may be empty).

%
%
%
%

Recall that if $X$ is an $R$-orientable topological space, a version of Poincar\'{e} duality says that
	\[ H_*(X ; R) \cong H^{\dim X - *}_c(X ; R). \]
In particular, for a stratified space we have the following.	
\begin{proposition} \label{lefschetz duality} If the strata of a stratified space $W$ are $R$-orientable, then there are isomorphisms 
	\[ H_*(M[p]; R) \cong H^{\dim(M[p]) - *}_{c}(M[p]; R). \]
\end{proposition}

\begin{lemma} \label{spectral sequence} If the strata of a stratified space $W$ are $R$-orientable of codimension $p$, then there is a spectral sequence
	\[ E^1_{pq} = H_q(M[p] ; R) \implies H^{\dim{W} - p - q}_c(W ; R). \] 
\end{lemma}
\begin{proof} 
In general, if $B \subset A$ is closed then there is a long exact sequence in compactly supported cohomology associated to the inclusion $B \rightarrow A$ of the form
	\[ \cdots \rightarrow H^*_c( A - B) \rightarrow H^*_c(A) \rightarrow H^*_c(B) \rightarrow H^{*+1}_c(A - B) \rightarrow \cdots. \]
Therefore, associated to the filtration
\[ W = F_0 \supset F_1 \supset \ldots \supset F_l = \emptyset, \]
there is a spectral sequence in compactly supported cohomology of the form
	\[H_c^{\dim W - p - q}(M[p];R) \implies H^{\dim W - p - q}_c(W ; R). \]
Note that this is the same spectral sequence that one would get for singular homology where $H_c^{\dim W - p - q}(M[p];R)$ plays the role of $H^{\dim W - p - q}(F_p, F_{p+1};R)$.

The proof now proceeds as in \cite[Lemma 3.3]{gj04}. In our case, the filtration of the singular compactly supported cochains of $W$ are given by 
	\[ F_p(C^{\dim{W} - \bullet}) := C^{\dim{W} - \bullet}_c(W, F_p). \]
The $E^1$ page of the spectral sequence is 
	\[ E^1_{pq} = H_c^{\dim{W} - p -q}(M[p]) \]
and it converges to
	\[ E^\infty_{pq} = H_c^{\dim W - p - q}(W). \]
The differentials $d$ of the $E^1$ page of the spectral sequence are given by the connecting homomorphism $\del$ in the long exact sequence in compactly supported cohomology for the triple $(F_{p-1}, F_p, F_{p+1})$ in the sense that the following diagram commutes.
\[ \xymatrix{ H_c^{\dim W - p - q}(M[p]) \ar[r]^{\del \;\;\;\;\;} \ar[d]_\cong &	H_c^{\dim W - p - q + 1}(M[p-1]) \ar[d]^{\cong} \\
	H_q(M[p]) \ar[r]^{d} &	H_q(M[p-1]).	} \]
The maps down are given by \cref{lefschetz duality}. Thus we can identify the $E^1$ page as
	\[  E^1_{pq} = H_q(M[p]). \]
\end{proof}

\begin{remark} We allow for the possibility that the filtration differences $F_p - F_{p+1}$ are empty in which case the duality theorem says that $0 \cong 0 $.
\end{remark}

\subsection{A stratification for symmetric complements}

Let $\lambda$ be a partition of $n$. Let $F_p(\overline{w}^c_\lambda (X))$ be the union of subspaces of $\overline{w}^c_\lambda (X)$ of codimension greater than or equal to $p$ of the form $w_{\lambda'}(X)$ for some partition $\lambda'$ of $n$. We have the following lemma.

\begin{lemma} $F_p(\bar{w}^c_\lambda) - F_{p+1}(\bar{w}^c_\lambda)$ is either empty or a disjoint union of spaces of the form $w_{\lambda'}$. 

\end{lemma}
\begin{proof} Let $M[p]$ be our notation for $F_p(\bar{w}^c_\lambda) - F_{p+1}(\bar{w}^c_\lambda)$. $M[p]$ will be non-empty only when $\dim X$ divides $p$. Assume we are now in this case. If $\lambda$ is a partition of $n$ then $M[p]$ is the subspace of $\bar{w}_\lambda^c$ having multiplicities prescribed by some $\lambda'$ such that $\lambda' \nleq \lambda$. Moreover being in $M[p]$ means the number of parts of $\lambda'$ is equal to $n - \frac{p}{\dim X}$. Here, the number of \emph{parts} of a partition is the sum of the entries in $\vec{\lambda}.$

Each partition that gives a multiplicity in $M[p]$ cannot be a subpartition of one of the others since they have the same number of parts. Call the partitions $\lambda_1, \ldots, \lambda_c$.

Thus $M[p]$ breaks up as a disjoint union
	\[ \bigsqcup_{i = 1}^c w_{\lambda_i} .\]
\end{proof}

The $w_{\lambda_i}$ can be thought of as coloured configuration spaces by \cref{colour comp}.

\begin{corollary} \label{strata} Let $X$ be a manifold. $\bar{w}^c_\lambda(X)$ is a stratified space where the strata are disjoint unions of coloured configuration spaces.
\end{corollary}

The spaces $\bar{w}^c_{1^j \lambda}$ and $\bar{w}^c_{1^{j+1} \lambda}$ will not always have the same number of codimension $p$ strata. Howerver, the following proposition shows that $\bar{w}^c_{1^j \lambda}$ and $\bar{w}^c_{1^{j+1} \lambda}$ will have the same number of \emph{low} codimension $p$ strata depending on $n$, $j$ and $X$.

\begin{proposition} \label{strata bound} For any $\lambda$ a partition of $n$, $\bar{w}^c_{1^j \lambda}(X)$ and $\bar{w}^c_{1^{j+1}\lambda}(X)$ have the same number of codimension $p$ components in their strata for $0 \leq p \leq \frac{n + j }{2} \dim X$.
\end{proposition}
\begin{proof} 
Assume first that $\dim X$ divides $p$ for otherwise both $\bar{w}^c_{1^j \lambda}$ and $\bar{w}^c_{1^{j+1}\lambda}$ will have no codimension $p$ strata.

Let $\mathrm{strat}_p(-)$ denote the set of components of $M[p]$. There is a function
	\[ \mathrm{strat}_p(\bar{w}^c_{1^j \lambda}(X)) \rightarrow \mathrm{strat}_p(\bar{w}^c_{1^{j+1}\lambda}(X)) \]
that sends $w_{\lambda'} \mapsto w_{1\lambda'}$. This map is clearly injective for all $i$. It is not surjective when $\mathrm{strat}_p(\overline{w}^c_{1^{j+1}\lambda}(X))$ contains a stratum of the form $w_{\lambda''}$ where $\lambda''$ is a partition without any $1$'s. 

We claim that $w_{\lambda''}$ has at least $n + j + 1 - \frac{2p}{\dim X}$ points of multiplicity 1. Since $w_{\lambda''}$ has codimension $p$ in $\bar{w}^c_{1^{j+1}}$, the maximum number of multiplicity 1 points that could have been removed comes from pairing $\frac{2p}{\dim X}$ points into multiplicity 2 points. Therefore there are at least $n + j + 1 - \frac{2p}{\dim X}$ points of multiplicity 1 left over. This occurs when $\lambda'' = 1^y2^z$, where $y = n+j + 1 - \frac{2p}{\dim X}$ and $z = \frac{p}{\dim X}$, if this is an allowed stratum i.e., if this $\lambda''$ is not a subpartition of $\lambda$. Thus the map is surjective when
	\[ n + j + 1 - \frac{2p}{\dim X} \geq 1, \]
or equivalently when 
	\[ p \leq (n+j)\frac{\dim X}{2}. \]
\end{proof}

\subsection{The stabilisation map}
Let $X$ be a connected open manifold of dimension greater than one that is the interior of a compact manifold with boundary. Then we have stabilisation maps $\conf_n(X) \rightarrow \conf_{n+1}(X)$. Similarly one can define a map $\symp^n (X) \rightarrow \symp^{n+1}(X)$ that induces maps 
	\[ \stab: \bar{w}^c_{1^j \lambda} \rightarrow \bar{w}^c_{1^{j+1} \lambda} \]
which adds a point of multiplicity one. The construction is analogous to our the construction in \cref{sec: stab maps}, replacing the ordered configuration space with the product $X^n$. Note that if a configuration in $\bar{w}^c_{1^j \lambda}$ was of an allowed multiplicity, the its image under the map $\stab$ is an allowed multiplicity in $\bar{w}^c _{1^{j+1} \lambda}$ since the newly added point will have multiplicity 1.

\begin{lemma} \label{stratmap} The map $\stab: \bar{w}^c_{1^j \lambda} \rightarrow \bar{w}^c_{1^{j+1} \lambda}$ on strata is the map
	 \[ M[p](\bar{w}^c_{1^j\lambda}) =: M[p]  \rightarrow M'[p] := M[p](\bar{w}^c_{1^{j+1}\lambda}). \] 
On the components of the strata, it is the stabilisation map on configuration spaces given by
	\[ \conf_{\vec{v}} \rightarrow \conf_{\vec{v} + e_1}. \]
\end{lemma}
\begin{proof} From the proof of \cref{strata bound}, it is clear that the stabilisation map sends $M[p]$ to $M'[p]$. Identifying $\conf_{\vec{v}}(X)$ as the subspace of $\symp^n(X)$ where $\vec{v}(i)$ corresponds to the number multiplicity $i$ elements, we see that adding a multiplicity one point is the same as adding one to the first entry of $\vec{v}$.
\end{proof}

\subsection{Homological stability for symmetric complements}

The last ingredient we will need for the proof of \cref{hstab symcomp} is the following lemma on spectral sequences.
\begin{lemma} \label{ssbound} Consider two first quadrant spectral sequences
	\[ E_{pq}^{1} \Rightarrow E_{pq}^{\infty} \]
	\[ 'E_{pq}^1 \Rightarrow {'E}_{pq}^\infty. \]
Let $f: E_{pq} \rightarrow {'E}_{pq}$ be a map of homological spectral sequences. If $f$ is an isomorphism on $E^1_{pq}$ for $p + q \leq 2* + 1$ then the induced map on $E^\infty$
	\[ f_\infty : E_{pq}^\infty \rightarrow {'E}_{pq}^\infty \]
is an isomorphism for $p + q \leq *$.
\end{lemma}
\begin{proof} In a first quadrant spectral sequence, 
	\[E^\infty_{pq} = E^{\max(p, q+1)+1}_{pq}.\]
This is because all the nonzero differentials going into the triangle $p + q \leq *$ come from the triangle $p + q \leq * + 1$. If all these differentials and modules are the same in both spectral sequences then $f$ will be an isomorphism on $E^\infty$ for $p + q \leq *$. 

Now suppose $f: E^1_{pq} \rightarrow {'E}^1_{pq}$ was an isomorphism for $p + q \leq 2* + 1$. Then the map on $E^2$ pages is an isomorphism for $p+q \leq 2*$ since all morphisms, images, kernels and maps are the same for both spectral sequences in the relevant region where differentials can come from and go to. Working our way through the pages of the spectral sequence we see that $E^{*+1}_{pq}$ and ${'E}^{*+1}_{pq}$ are isomorphic in the range $p+q \leq *+1$. Thus the $E^\infty$ pages are isomorphic for $p+q \leq *$.
\end{proof}

\begin{proof}[Proof of \cref{hstab symcomp}]
Consider 
	\[ \stab: \bar{w}^c_{1^j \lambda} \rightarrow \bar{w}^c_{1^{j+1}\lambda}. \]

We use the stratifications for $\bar{w}^c_{1^j \lambda}$ and $\bar{w}^c_{1^{j+1}\lambda}$ as in \cref{strata}. The filtration for $\bar{w}^c_{1^{j+1} \lambda}$ will be denoted by $F_p'$.

By \cref{spectral sequence}, there are spectral sequences
	\[ E^1_{pq} = H_q(M[p]) \Rightarrow H^{\dim(\bar{w}^c_{1^j \lambda}) - p - q}_c(\bar{w}^c_{1^j \lambda} )\]
and
	\[ 'E^1_{pq} = H_q(M'[p]) \Rightarrow H^{\dim(\bar{w}^c_{1^{j+1}\lambda})  - p - q}_c(\bar{w}^c_{1^{j+1} \lambda}) \]

Recall from the proof of \cref{spectral sequence} that the differentials of the spectral sequence come from the connecting map of the long exact sequence in cohomology for the triple $(F_{p-1}, F_{p}, F_{p+1})$. Naturality of the connecting map in the long exact sequence implies that the map stab induces a morphism of spectral sequences. On each component, the map 
	\[ M[p]  \rightarrow M'[p] \]
is the stabilisation map on coloured configuration spaces by \cref{stratmap}. We assume now that $\dim X$ divides $p$ since we will be interested in when this map is an isomorphism. 


As in the proof of \cref{strata bound}, the strata $M[p]$ break up into a disjoint union of coloured configuration spaces where there are at least $j + n - \frac{2p}{\dim X}$ points of multiplicity $1$ in each component. When $j + n - \frac{2p}{\dim X} \geq 0$ the stabilisation map $M[p] \rightarrow M'[p]$ is surjective on the set of components by \cref{strata bound}.

Using \cref{hstab coloured}, we see that the map
	\[ \stab_* : E^1_{pq} \rightarrow (E^1_{pq})' \]
is an isomorphism for $q + \frac{p}{\dim X} \leq (j+n)/2$ which certainly implies that $\stab_*$ is an isomorphism for
	\[ q + p \leq \frac{j+n}{2}. \]
By \cref{ssbound} this means that the map
	\[ E^\infty_{pq} \rightarrow (E^\infty_{pq})' \]
is an isomorphism for $p + q \leq \frac{j+n}{4} - \frac{1}{2}$.

Lastly observe that the spaces $\bar{w}_{\lambda}^c$ satisfy rational Poincar\'{e} duality since they are open subspaces of the orbifold $\symp^n(X)$.  Therefore the map 
	\[\stab: H_*(\bar{w}_{1^j\lambda}^c; \mathbb{Q}) \rightarrow H_*(\bar{w}_{1^{j+1}\lambda}^c, \mathbb{Q})\] is an isomorphism for $* \leq \frac{j+n}{4} - \frac{1}{2}$. 
 \end{proof}

\bibliographystyle{alpha}
\bibliography{references}

\begin{thebibliography}{VMW12}

\bibitem[Arn69]{arnold69}
V.~I. Arnol'd.
\newblock The cohomology ring of the group of dyed braids.
\newblock {\em Mat. Zametki}, 5:227--231, 1969.

\bibitem[Chu12]{church12}
Thomas Church.
\newblock Homological stability for configuration spaces of manifolds.
\newblock {\em Invent. Math.}, 188(2):465--504, 2012.

\bibitem[Dol62]{Dold62}
Albrecht Dold.
\newblock Decomposition theorems for s(n)-complexes.
\newblock {\em Annals of Mathematics}, 75(1):pp. 8--16, 1962.

\bibitem[GJ94]{gj04}
Ezra Getzler and John J.~D. Jones.
\newblock Operads, homotopy algebra and iterated integrals for double loop
  spaces.
\newblock 1994.
\newblock preprint: arXiv:hep-th/9403055.

\bibitem[KM13a]{km13b}
Alexander Kupers and Jerremy Miller.
\newblock Homological stability for complements of closures.
\newblock 2013.
\newblock Private communication.

\bibitem[KM13b]{km13}
Alexander Kupers and Jerremy Miller.
\newblock Homological stability for topological chiral homology and
  completions.
\newblock 2013.
\newblock preprint: arXiv:1311.5203.

\bibitem[McD75]{mcduff75}
Dusa McDuff.
\newblock Configuration spaces of positive and negative particles.
\newblock {\em Topology}, 14:91--107, 1975.

\bibitem[Pal11]{palmer11}
Martin Palmer.
\newblock Homological stability for oriented configuration spaces.
\newblock 2011.
\newblock preprint: arXiv:1106.4540.

\bibitem[RW13]{orw13}
Oscar Randal-Williams.
\newblock Homological stability for unordered configuration spaces.
\newblock {\em Q. J. Math.}, 64(1):303--326, 2013.

\bibitem[Seg73]{segal73}
Graeme Segal.
\newblock Configuration-spaces and iterated loop-spaces.
\newblock {\em Invent. Math.}, 21:213--221, 1973.

\bibitem[Seg79]{segal79}
Graeme Segal.
\newblock The topology of spaces of rational functions.
\newblock {\em Acta Math.}, 143(1-2):39--72, 1979.

\bibitem[VMW12]{vw12}
Ravi Vakil and Melanie Matchett~Wood.
\newblock Discriminants in the grothendieck ring.
\newblock 2012.
\newblock preprint: arXiv:1208.3166.

\end{thebibliography}

\end{document}